\documentclass{amsart}

\usepackage{amssymb,amsmath,amsthm}
\usepackage{graphics}
\usepackage[arrow, matrix, curve]{xy}
\usepackage{txfonts}
\usepackage{bbm}

\theoremstyle{definition}

\theoremstyle{plain}
\newtheorem{theorem}{Theorem}

\newtheorem{lemma}[theorem]{Lemma}

\theoremstyle{remark}

\theoremstyle{dotless}

\newcommand{\N}{\mathbb{N}}

\newcommand{\Z}{\mathbb{Z}}

\renewcommand{\phi}{\varphi}
\renewcommand{\rho}{\varrho}

\title{Arithmetic Progressions in  Abundance by Combinatorial Tools}

 \author[M.\ Beiglb\"ock ]{Mathias Beiglb\"ock }

\address{Institut f\"ur Diskrete Mathematik und Geometrie, Technische Universit\"at Wien\endgraf Wiedner Hauptstra\ss e 8-10/104\\ 1040 Wien, Austria} \email{mathias.beiglboeck@tuwien.ac.at}

\thanks{The  author gratefully acknowledges financial support from the Austrian Science Fund (FWF) under grant S9612.}

\subjclass[2000]{05D10}
\keywords{arithmetic progressions, piecewise syndetic sets, van der Waerden's Theorem}
\begin{document}

\begin{abstract} Using the algebraic structure of the  Stone-\v Cech compactification of the integers, Furstenberg and Glasner proved that for arbitrary $k\in\N$, every piecewise syndetic set contains a piecewise syndetic set of $k$-term arithmetic progressions. 

We present a purely combinatorial argument which allows to derive this result directly from van der Waerden's Theorem. 

\end{abstract}

\maketitle

A subset $S$ of $ \Z$ is piecewise syndetic\footnote{piecewise syndetic sets are called \emph{big} in \cite{FuGl98}.} if there exists $r\in \N$ such that $\bigcup_{t=1}^r S-t$ contains arbitrarily long intervals. A subset $S$ of a countable abelian group  $(G,+)$ is piecewise syndetic if there exists a finite set $F$ such that $\bigcup_{t\in F} S-t$ is \emph{thick} in the sense that it contains a shifted copy of any finite subset of $G$. 
 
Given $k,r \in\mathbb N$, van der Waerden's Theorem \cite{Waer27} states that one cell of any partition $\{C_1,\ldots,C_r\}$ of a long enough interval of integers contains a $k$-term arithmetic progression. Since arithmetic progressions are invariant under shifts, it follows that every piecewise syndetic set contains arbitrarily long arithmetic progressions. 

It is a simple combinatorial exercise to show that whenever a piecewise syndetic set is devided into  finitely many parts, one cell is piecewise syndetic itself. In particular  the existence of arbitrarily long arithmetic progressions in piecewise syndetic sets implies  that one cell of any finite partition of the integers  contains arbitrarily long arithmetic progressions.  

Taking the above facts into account, we see that Theorem \ref{vdW} is just a slightly disguised version of van der Waerden's Theorem.
\begin{theorem}\label{vdW}
Let $k\in\mathbb \N$ and assume that $S\subseteq \mathbb Z$ is piecewise syndetic.  Then 
\begin{align}\label{OrdinaryVdW}
\{(a,d): a, a+d, \ldots ,a+kd \in S, d\neq 0\}\neq \emptyset.
\end{align}
\end{theorem}

Furstenberg and Glasner \cite{FuGl98} used the algebraic structure which may be imposed on the Stone-\v Cech compactification of $\Z$ to prove that every piecewise syndetic set contains a piecewise  syndetic set of arithmetic progressions.

\begin{theorem}\label{AbundanceVdW}
Let $k\in\mathbb \N$ and assume that $S\subseteq \mathbb Z$ is piecewise syndetic.  Then 
\begin{align}\label{OrdinaryVdW}
\{(a,d): a, a+d, \ldots ,a+kd \in S\}
\end{align}
is piecewise syndetic in $\Z^2$.
\end{theorem}
The purpose of this paper is to show how Theorem \ref{AbundanceVdW} can be derived from van der Waerden's Theorem by elementary tools, in particular, without using the axiom of choice. Note that in contrast to \cite{FuGl98} our argument does not provide a new proof of van der Waerden's Theorem.

We will employ the following simple lemma which we do not prove.
\begin{lemma}
If $M\subseteq \mathbb Z^2$  is piecewise syndetic and $x,y \in \mathbb Z, z\in \mathbb Z\setminus \{0\}$ then 
\begin{align}
\{(A_1+xA_2+y, z A_2): (A_1,A_2)\in M\}\end{align} is piecewise syndetic as well.
\end{lemma}

\begin{proof}[Proof of Theorem \ref{AbundanceVdW}]
Choose $r\in \mathbb N$ such that $\bigcup_{t=1}^r S-t $ is thick, i.e.\ contains arbitrarily long intervals. By van der Waerden's Theorem there exists $K\in\mathbb N$ such that whenever $\{A, A+D,\ldots, A+KD\}, A,D\in \mathbb Z, D\neq 0$ is partitioned into $r$ cells, one cell contains  a $(k+1)$-term arithmetic progression. 
Using the assumption that $\bigcup_{t=1}^r S-t $ is thick, one easily verifies  that the set 
\begin{align}
  B:=\Big\{(A,D)\in \mathbb Z^2: A, A+D, \ldots , A+KD\in \bigcup_{t=1}^r S-t\Big\}
\end{align}
is thick in $\mathbb Z^2$.

Given an arithmetic progression  $\{A,A+D,\ldots, A+KD\}\subseteq \bigcup_{t=1}^r S-t $ there exists some $t\in \{1,\ldots , r\}$ such that  $\{A, A+D,\ldots, A+KD\} \cap (S-t)$ contains a $(k+1)$-term arithmetic progression.
 Thus we can assign to each pair $(A,D)\in B$ a triple $(\alpha,\delta,t)=:\phi(A,D)\in \{0,\ldots,K\}\times \{1,\ldots,K\}\times\{1,\ldots,r\}$, such that  
\begin{align}
(A+\alpha D) + i\cdot (\delta D)\in S-t 
\end{align}
for $i=0,\ldots, k$. The function $\phi$  constitutes a finite coloring of $B$.  Thus there exist $\alpha\in\{0,\ldots,K\}, \delta\in\{1,\ldots,K\}$ and $ t\in \{1,\ldots,r\} $ such that 
\begin{align}
M:=\{(A,D)\in B: \phi(A,D)=(\alpha,\delta,t)\}
\end{align}
 is piecewise syndetic. By definition of $\phi$, $(A+\alpha D) + i\cdot (\delta D) \in S-t$  for all $i\in\{0,\ldots,k\}$ and $(A,D)\in M$. Set 
\begin{align}
\widetilde M:=\{(A+\alpha D +t , \delta D): (A,D)\in M\}.   
\end{align}
Piecewise syndeticity of $M$ implies piecewise syndeticity of $\widetilde M$ and $a, a+d, \ldots, a+kd\in S$ for all $(a,d)\in \widetilde M$.
\end{proof}

In \cite{BeHi01,HiLS02}  Theorem \ref{AbundanceVdW} is extended to more general (semi-)groups and various other notions of largeness using ultrafilter techniques. We don't know whether the combinatorial argument presented here can be enhanced in a way which makes it applicable to a more abstract setting. 

\def\ocirc#1{\ifmmode\setbox0=\hbox{$#1$}\dimen0=\ht0 \advance\dimen0
  by1pt\rlap{\hbox to\wd0{\hss\raise\dimen0
  \hbox{\hskip.2em$\scriptscriptstyle\circ$}\hss}}#1\else {\accent"17 #1}\fi}

\end{document}